\documentclass[11pt]{article}

\usepackage{mathtools,amssymb}
\usepackage{amsfonts}
\usepackage{amsmath}
\usepackage{amsthm}
\usepackage{graphicx}
\usepackage{caption}
\usepackage{natbib}
\usepackage{color}
\usepackage{setspace}
\usepackage{natbib}
\usepackage{datetime2}

\usepackage[T1]{fontenc}

\usepackage{times}
\usepackage[left=2cm,right=2cm,top=2.5cm,bottom=2.5cm]{geometry}

\RequirePackage[colorlinks,citecolor=blue,urlcolor=blue,backref=page]{hyperref}

\newcommand{\remove}[1]{}
\newcommand{\no}{\nonumber}

\newcommand{\pr}[1]{\mathbb{P}\left\{ #1 \right\}}

\newcommand{\EXP}[1]{\mathbb{E}\!\left[#1\right] }

\newcommand{\VAR}[1]{\mathsf{VAR}\!\left(#1\right) }

\newtheorem{theorem}{Theorem}[section]
\newtheorem{corollary}[theorem]{Corollary}
\newtheorem{lemma}[theorem]{Lemma}
\newtheorem{proposition}[theorem]{Proposition}
\newtheorem{remark}[theorem]{\bf Remark}

\newcommand{\1}{{\bf 1}}
\newcommand{\bi}{{\bf i}}

\newcommand{\R}{\mathbb{R}}
\newcommand{\N}{\mathbb{N}}

\newcommand{\ER}{Erd\"{o}s-R\'enyi}

\newcommand{\reg}{\text{reg}}

\title{Edge ideals of \ER \, random graphs : Linear resolution, unmixedness and regularity.}
\author{ Arindam Banerjee\footnote{Department of Mathematics, Ramakrishna Mission Vivekananda Educational and Research Institute, 123.arindam@gmail.com} \, and  \, D. Yogeshwaran \footnote{Theoretical Statistics and Mathematics Unit, Indian Statistical Institute Bangaluru, India. d.yogesh@isibang.ac.in}}

\begin{document}

\maketitle 

\begin{abstract}
We study the homological algebra of edge ideals of \ER \, random graphs. These random graphs are generated by deleting edges of a complete graph on $n$ vertices independently of each other with probability $1-p$.  We focus on some aspects of these random edge ideals - linear resolution, unmixedness and algebraic invariants like the Castelnuovo-Mumford regularity,  projective dimension and depth.  We first show a double phase transition for existence of linear presentation and resolution and determine the critical windows as well. As a consequence,  we obtain that except for a very specific choice of parameters (i.e., $n,p := p(n)$), with high probability, a random edge ideal has linear presentation if and only if it has linear resolution. This shows certain conjectures hold true for large random graphs with high probability even though the conjectures were shown to fail for determinstic graphs.  Next,  we study asymptotic behaviour of some algebraic invariants - the Castelnuovo-Mumford regularity,  projective dimension and depth -  of such random edge ideals in the sparse regime (i.e., $p = \frac{\lambda}{n}, \lambda \in (0,\infty)$).  These invariants are studied using local weak convergence (or Benjamini-Schramm convergence) and relating them to invariants on Galton-Watson trees.  We also show that when $p \to 0$ or $p \to 1$ fast enough, then with high probability the edge ideals are unmixed and for most other choices of $p$,  these ideals are not unmixed with high probability. This is further progress towards the conjecture that random monomial ideals are unlikely to have Cohen-Macaulay property \cite{PetHa2017,deLoera2019average} in the setting when the number of variables goes to infinity but the degree is fixed.  
	
\end{abstract}


\vspace{0.1cm}

\section{Introduction}
\label{s:intro}

\hspace{0.1 in}   
Square-free monomial ideals and Stanley-Reisner ideals of flag complexes have emerged as two important subtopics within combinatorial commutative algebra \cite{Stanley2007,Miller2004}.  In the last decade, a specific class of square-free monomial ideals called {\em the edge ideals} have garnered significant attention (see \cite{van2013}). These ideals are generated by edges of a simple graph and various results about these edge ideals displaying the interplay between the algebraic and the combinatorial properties have been proven in the recent years. In this article, we demonstrate that bringing probabilistic angle to this algebra combinatorics interplay has the potential of generating many interesting results and shedding new light.  \\

Motivated by the success of probabilistic methods in the study of combinatorial structures, especially in the now classical subject of random graphs \cite{Frieze2016} and more recently in random simplicial complexes \cite{Kahle14survey},  it is natural to bring these techniques to combinatorial commutative algebra. One particular advantage of bringing in probabilistic techniques to the study of edge ideals is to answer questions about typical or predominant behaviour of edge ideals of large graphs. Often, one only understands behaviour of a sub-class of edge ideals but with probabilistic ideas, one might be able to understand typical behaviour of edge ideals on large graphs or behaviour of edge ideals on most large graphs. However, there have been very few studies in this direction. Recently, Erman and Yang \cite{Erman17} investigated graded Betti numbers of Stanley Reisner ideals of random flag complexes. They have demonstrated certain Betti numbers satisfy some asympotic convergence result when the number of vertices go the infinity. Further, the related asymptotic syzygies have been observed to have various interesting properties. As many other important algebraic invariants like regularity, depth and projective dimension are intimately related with Betti numbers, one expectedly gets certain asymptotics for these too. In another set of recent works \cite{PetHa2017,deLoera2019average,silverstein2020asymptotic},  a different model of random monomial ideals was investigated in detail. Here for asymptotics they study behaviors of different homological invariants when the degree of the generators  (instead of the number of generators) go to infinity. In particular, the Hilbert functions, Krull dimensions and graded Betti numbers of random monomial ideals are studied and interestingly, it is shown that for most choices of parameters the random monomial ideal is not Cohen-Macaulay with high probability.\\

In this article, we investigate edge ideals of \ER \, random graphs.  Edge ideals are an important class of monomial ideals that are more tractable due to their graph-theoretic connections. Algebraically speaking edge ideals cover the same class of ideal as the Stanley-Reisner ideals of flag complexes, namely the squarefree quadratic monomial ideals.  One of our main contributions (Theorem \ref{t:reginp})  is the introduction of local weak convergence and random trees in the study of random edge ideals.  It is known that large \ER \, random graphs with bounded average degree ``locally resemble" a Galton-Watson tree and since homological invariants of the edge ideals of trees are comparatively well understood,  this allows us to prove laws of large numbers for Castelnuovo-Mumford regularity and projective dimension (and as an immdediate consequence for depth due to the Auslander-Buchbaum Theorem).  This is an improvement of Corollary 5.2 in Erman and Yang \cite{Erman17}.  More detailed comparisons with literature are pointed out after the respective theorem statements.  \\

Another main contribution (Theorems \ref{t:lrlp} and \ref{t:lrlplambda}) of the article is determining sharp thresholds for existence of linear minimal free resolution for random edge ideals and their powers.  If an edge ideal has linear resolution so does all its powers \cite{HeHiZ2004}.  Fr\"oberg characterised edge ideals with linear resolution as those whose underlying graphs are co-chordal (i.e., the complement graph has no induced cycle other than triangles) \cite[Theorem 1.1]{Fr1990}.  Motivated by this,  it was asked by Chris Francisco, Tai Huy Ha and Adam Van Tuyl whether all powers of an edge ideal from second power onwards have linear resolution when the complement graph has no induced $4$-cycle.  Based on various Macaulay2 examples computed by Chris Francisco, answer to this question appeared affirmative.  However Eran Nevo and Irena Peeva gave a counter example to this fact \cite{NePe2013} and modified the question to add a condition on Castelnuovo-Mumford regularity for expecting the same conclusion (see \cite{NePe2013}) for the details of this direction of research).  As examples are abundant where edge ideals do not have linear resolution but all higher powers have linear resolution ($5$-cycle for example),  this question by Nevo-Peeva had drawn lots of attention and many classes of graphs where powers of edge ideals have linear resolution second power onwards.  These studies motivated us to understand the behaviour of `typical' large graphs.  \ER \, random graphs gives us a natural model for such a study.  We show that `typical' edge ideals behave even better (edge ideals already behave better than general monomial ideals as mentioned above).  That is to say that,  as the number of vertices go to infinity the asymptotic probability of an edge ideal having linear resolution and linear presentation (which is equivalent to having no $4$-cycle in the complement) are equal.  This shows that not only does the Nevo-Peeva question has a positive answer with high probability in the random set up,  the original question of Francisco-Ha-Van Tuyl also has an affirmative answer in the random set up.  Apart from linear presentation and resolution,  asymptotics of algebraic invariants (regularity,  projective dimension and depth),  we also study unmixedness of random edge ideals (Theorem \ref{t:unmixed_Inp}).  Here,  we show that for most choice of parameters the random edge ideal is not unmixed with high probability.  More specifically,  our results indicate that random edge ideals should be unmixed if and only if $p \to 0$ or $p \to 1$ fast enough. \\

Our proofs use a range of probabilistic ideas such as sharp thresholds for containment of subgraphs, Poisson approximation for cycle counts,  concentration inequality for near-Lipschitz functionals and local weak convergence (i.e., Benjamini-Schramm convergence) theory.  From the combinatorial commutative algebra side,  we use various combinatorial characterizations of algebraic properties of the edge ideals. as well as prove some polynomial Lipschitz properties for algebraic invariants.  En route to our results,  we also prove a new result in random graph theory,  namely that of determining the critical windows for existence of $4$-cochordality and cochordality.  The relation of these properties to local cochordality and local $4$-cochordality are also investigated.  Also, our work gives strong motivation for further studies in random graphs, especially on induced matching number and minimal vertex covers.  Our proof of asymptotics of regularity and projective dimension in the sparse regime emphasizes the need to understand other algebraic invariants (for example,  regularity and projective dimension for higher powers of the ideal,  Betti numbers) on trees or their behaviour under vertex deletions in arbitrary graphs.  Our proof techniques also apply to study of edge ideals on other models of random graphs and could also be useful in study of more general random monomial ideals.  We refer the interested reader to the end of Section  \ref{s:intro} for more on further questions for research.   \\

In recent years, topological invariants of random simplicial complexes have received a lot of attention (see \cite{Kahle14survey}) but study of algebraic invariants of random graphs or simplicial complexes is still in its infancy. We hope that our work will complement the recent works of  \cite{Erman17,PetHa2017,deLoera2019average,booms2020heuristics,
silverstein2020asymptotic} on random monomial ideals and lead to a more fruitful interaction between probability and combinatorial commutative algebra. As expected our results complement well those of \cite{Erman17}. We study edge ideals of random graphs but \cite{Erman17} studies the Stanley Reisner ideals of random flag complexes, which are also square-free quadratic monomial ideals. The main difference between our approach and that of \cite{Erman17}  is that they relate algebraic invariants to the combinatorial properties of the flag complex but we relate the algebraic invariants to graph properties. These two are connected by the minimal vertex cover problem which is extremely hard. We also wish to clarify an important point of difference between the various studies. While \cite{Erman17,booms2020heuristics} study the case where the number of variables increases to infinity but  \cite{PetHa2017,deLoera2019average,silverstein2020asymptotic} study the case where the degree of the generators increases to infinity. The former approach is more natural from the viewpoint of random graph theory and this is the one we shall be taking in this article. \\

The rest of the article is organized as follows : In the following three subsections, we shall state our main results on linear resolution (Section \ref{sec:lplr}), algebraic invariants (Section \ref{s:reg}) and unmixedness (Section \ref{s:unmix}) as well as discuss them in the context of existing literature.  We provide the necessary algebraic, combinatorial and probabilistic preliminaries in Section \ref{sec:prelims}. The three subsections of Section \ref{sec:EIRG} shall each respectively prove the results stated in the next three subsections. 
 
 \subsection{Our Results} First, we quickly introduce edge ideals of \ER \, random graphs. For a finite simple (undirected) graph $G$ with vertex set $[n] := \{1.\ldots,n\}$ and edge set $E(G)$, we define the edge ideal as 
\begin{equation}
\label{d:ei}
I(G) := (x_ix_j \mid (i,j) \in E(G)) \subset K[x_1,\ldots,x_n],
\end{equation}
where $K[x_1,\ldots,x_n]$ denotes the polynomial ring over a fixed field $K$.\\

We shall now  introduce the \ER \, random graph, the simplest model for random graphs (undirected). The \ER \, random graph is denoted by $G(n,p)$ where $n \geq 1$ and $p \in [0,1]$. The vertex set of $G(n,p)$ is $[n] := \{1,\ldots,n\}$ and each pair $i,j$ ($i,j \in [n], i \neq j$) is an edge in $G(n,p)$ with probability $p$ and independent of other edges. More precisely, given a graph $G$ with vertex set as $[n]$ and having exactly $m$ edges, 
\begin{equation}
\label{d:erg}
\pr{G(n,p) = G} =  p^m(1-p)^{\binom{n}{2}-m}.
\end{equation}
Thus, every graph with $m$ edges has the same probability of being selected. The edge set is denoted by $E(n,p)$. When $p > 1$, by $G(n,p)$ we mean $G(n, \min\{p,1\})$. Often one studies properties of $G(n,p)$ as $n \to \infty$ and $p$ also varies with $n$. It is customary to write $p$ even though one is to understand that $p := p(n)$. By properties of $G(n,p)$, we mean graph properties i.e., those properties of graphs that are invariant under isomorphims. Further, we shall often say that for some specified sequence $p := p(n)$, $G(n,p)$ has the property $\mathcal{P}$ with high probability (or w.h.p. in short) if $\pr{\mbox{$G(n,p)$ has property $\mathcal{P}$}} \to 1$ as $n \to \infty$. For further definitions and unexplained notions in the rest of the subsection, we refer the reader to Section \ref{sec:prelims}. We shall always abbreviate the random edge ideal $I(G(n,p))$ by $I(n,p)$ for convenience.

\subsubsection{Critical windows for linear presentation and resolution of $I(n,p)$}
\label{sec:lplr}
 We first show a double phase transition for existence of  linear presentation and resolution for $I(n,p)$ as well as characterize the critical windows. Since linear presentation and resolution are not monotonic functionals of the underlying graphs, it is firstly not obvious that they exhibit a phase transition and let alone, a double phase transition. Further, we shall compute the exact probabilities for existence of linear presentation and resolution asymptotically in all the cases and surprisingly barring one very specific choice of parameters, the probabilities for the two coincide.  An important consequence of our two theorems below is that barring one particular choice of parameters $n,p$, w.h.p. $I(n,p)$ has linear resolution whenever it has linear presentation.  Our results also imply that w.h.p. all powers of $I(n,p)$ have linear resolution whenever they have linear presentation \cite{HeHiZ2004} (see Question 1.9 and Counter Example 1.10 of \cite{NePe2013}). Such a phenomenon only holds for bipartite graphs in the ``deterministic'' set up and is known to be false for general graphs.  In fact, it fails for the $5$-cycle itself.   A conjecture by Nevo and Peeva states that if $I(G)$ has regularity less than or equal to three and linear presentation then all higher powers have linear resolutions \cite[Open Problem 1.11(2)]{NePe2013}. This is known to be false without the first condition. Our two theorems below imply that the conjecture holds asymptotically for almost all graphs except for a very specific choice  of $n,p$.  We also discuss later the connection between local linear presentation and linear resolution (see Proposition \ref{p:lcchord} and Remark \ref{r:lcchordcchord}). 

Here is our first theorem that formalizes much of what was discussed above. 
\begin{theorem}
\label{t:lrlp}
$$ \lim_{n \to \infty} \pr{\mbox{$I(n,p)$ has linear resolution}} = 
\lim_{n \to \infty} \pr{\mbox{$I(n,p)$ has linear presentation}} =
\begin{cases} 
1 \, \, \, \, \mbox{if \, $(n(1-p))^4p^2 \to 0$}, \\
0 \, \, \, \, \mbox{if \, $(n(1-p))^4p^2 \to \infty$}.\\
\end{cases}
 $$
Further, if $p := p(n)$ is a sequence such that $\lim_{n \to \infty} n(1-p)  \in \{0,\infty\}$, then we have that
$$\lim_{n \to \infty} \pr{\mbox{$I(n,p)$ doesn't have a linear resolution but has a linear presentation}} = 0.$$
\end{theorem}
The above theorem leaves open the case of when $(n(1-p))^4p^2 \to \lambda \in (0,\infty)$ and our next theorem shall address this case. Here the probabilities shall depend on whether $p \to 0$ or $p \to 1$ and we shall see that the two probabilities do not coincide asymptotically when $p \to 1$ and $n(1-p) \to \lambda \in (0,\infty)$. This shows that the assumption in the second statement of the above theorem is very much necessary. This is the only case where the probabilities for linear presentation and linear resolution differ. 
\begin{theorem}
\label{t:lrlplambda}
Let $(n(1-p))^4p^2 \to \lambda \in (0,\infty)$. If $p \to 0$ then we have that
\begin{equation}
\label{e:cvgprlrlp1}
\lim_{n \to \infty} \pr{\mbox{$I(n,p)$ has linear resolution}} = \lim_{n \to \infty} \pr{\mbox{$I(n,p)$ has linear presentation}} = e^{-\frac{\sqrt{\lambda}}{2}}(1 + \frac{\sqrt{\lambda}}{2}).
\end{equation}
If $p \to 1$, then we have that
\begin{equation}
\label{e:cvgprlrlp2}
\lim_{n \to \infty} \pr{\mbox{$I(n,p)$ has linear presentation}} = e^{-\lambda/8} \, \, 
 \mbox{and} \, \, \lim_{n \to \infty} \pr{\mbox{$I(n,p)$ has linear resolution}} = e^{-\sum_{k \geq 4} \frac{\lambda^{k/4}}{2k}}.
 \end{equation}
\end{theorem}
\begin{remark}
\label{r:ermanyangcomp1}
We now compare one of the consequences of the above theorems with those in \cite{Erman17}. We first recall that an edge ideal has linear resolution if and only if it has regularity $2$ (see \cite{NePe2013}). For an ideal $I$, we denote the Castelnevo-Mumford regularity of the ideal by $\reg(I)$. As a corollary of \cite[Theorem 1.6]{Erman17}, we have that $\reg(I(n,p)) \geq s + 2$ w.h.p. for $\frac{1}{n^{1/s}} << p << \frac{1}{n^{1/(s+1)}}$\footnote{$a_n << b_n$ means that $\frac{a_n}{b_n} \to 0$ as $n \to \infty$.}. So, if we set $s = 1$ then we obtain that $\reg(I(n,p)) > 2$ w.h.p. for $p >> \frac{1}{n}$. However, from Theorem \ref{t:lrlp}, we have that $\reg(I(n,p)) = 2$ for $p << \frac{1}{n^2}$, $\reg(I(n,p)) > 2$ for $\frac{1}{n^2} << p $ and $1 - p >> 1/n$ and $\reg(I(n,p)) = 2$ for $1 - p << 1/n$. In other words, we have exactly determined the parameter regime where $\reg(I(n,p)) = 2$ and thus significantly improving the bounds obtained using the results of \cite{Erman17}. We shall later show that $reg(I(n,p))$ grows linearly in $n$ when $p = \frac{\lambda}{n}$ for $\lambda \in (0,\infty)$. See Theorem \ref{t:reginp} and Remark \ref{r:ermanyangcomp2}.
\end{remark}
 We now describe the main proof ideas involved in the two theorems.  By Froberg's theorems (Theorems \ref{thm:Froberg} and \ref{thm:Froberg1}),  linear resolution and presentation are equivalent to  cochordality and $4$-cochordality respectively.  It is to be expected of random graphs that if $4$ cycles have a chord with high probability than so will higher order cycles.  We show the same via computing first and second moments of cycles of length $k$ without a chord in the complement and this yields the proof of Theorem \ref{t:lrlp}. The complement of $G(n,p)$ is $G(n,1-p)$ and this is of crucial use in all our proofs.  As for Theorem \ref{t:lrlplambda},  in the regime when $p \to 0$,  we know from random graph theory that the graph is nothing but a collection of disjoint edges and the distribution of the number of edges is also known.  In the regime when $p \to 1$, for all $k \geq 1$ we know the joint distribution of number of cycles of order upto $k$ in the complement graph. The limiting distribution in both the cases is the Poisson distribution with appropriate parameters and exploiting this along with some approximation for cycles without a chord,  we derive the exact probabilities in these cases.  A consequence of our proofs is that similar results as Theorem \ref{t:lrlp} and \ref{t:lrlplambda} hold for $4$-cochordality and cochordality. To the best of our knowledge,  such a phase transition result for cochordality is not available in the random graph literature.
 
\subsubsection{Algebraic invariants in the sparse regime}
\label{s:reg}

We now describe the asymptotic behaviour of some algebraic invariants in the sparse regime.   Recall that Castelnevo-Mumford regularity of the ideal is denoted by $\reg(I)$,  projective dimension by $\text{pd }I$ and depth by $\text {depth }I.$  Regularity is a measure of complexity of an ideal and it achieving the minimum possible value indicates that the ideal has linear resolution.  Projective dimension and depth are two closely related yet different measures of size (the first is the length of the minimal free resolution and the second is the maximum size of a regular sequence). For a graph $G$,  we denote the induced matching number by $\nu(G)$ and by $|G|$, we mean the number of vertices in $G$. Further, $GW(\lambda)$ denotes the Galton-Watson tree with Poisson($\lambda$) offspring distribution (see Section \ref{s:probprelims} for definition).  
\begin{theorem}
\label{t:reginp}
Consider the \ER \, random graph $G(n,p)$ with $p = \lambda/n$ for $\lambda > 0$. Then we have the following :
\begin{enumerate}
\item It holds that
\begin{equation}
\label{e:reglimit}
\lim_{n \to \infty} \frac{\reg(I(n,p)) - \EXP{reg(I(n,p))}}{n} = 0, \, \, \mbox{a.s.}.
\end{equation}
The above statement also holds for projective dimension $\text{pd }I(n,p)$ and $\text{depth }I(n,p)$. 

\item For $\lambda \leq 1$, we have that as $n \to \infty$, 
$$ n^{-1}\reg(I(n,p)) \to  \EXP{\frac{\nu(GW(\lambda))}{|GW(\lambda)|}} \, \, \, \, \mbox{a.s.},$$
$$ n^{-1}\text{pd }I(n,p) \to  \EXP{\frac{\text{pd }GW(\lambda)}{|GW(\lambda)|}}\, \, \, \, \mbox{a.s.},$$
and
$$ n^{-1}\text{depth }I(n,p) \to  \EXP{\frac{\text{depth }GW(\lambda)}{|GW(\lambda)|}}\, \, \, \, \mbox{a.s.}.$$
\end{enumerate}
\end{theorem}
\begin{remark}
\label{r:ermanyangcomp2}
\begin{enumerate}
\item Using \eqref{e:reglimit}, we can derive the following bounds for the growth rates of regularity :  There exists a positive constant $\beta_{\infty}(\lambda) \in (e^{-\lambda},\infty)$ (see \eqref{e:betainf} for an implicit definition) such that 
\begin{equation}
\label{e:regbds}
\beta_{\infty}(\lambda) - e^{-\lambda} \leq \liminf \frac{\reg(I(n,p))}{n} \leq \limsup \frac{\reg(I(n,p))}{n} \leq  1 - \frac{t_*  + e^{-\lambda t_*} + \lambda t_* e^{-\lambda t_*}}{2} \, \,  \mbox{a.s.,}
\end{equation}
where $t_*$ is the smallest root of $t = e^{-\lambda e^{-\lambda t}}$. Since $\nu(G(n,p)) \geq \hat{\beta}_0(G(n,p))$, with $\hat{\beta}_0(G(n,p))$ being the number of non-trivial components, the lower bound follows trivially from the strong law for $\hat{\beta}_0(G(n,p))$ (see \eqref{e:betahinf}). For the upper bound, we bound $\nu(G(n,p))$ by the matching number $M(G(n,p))$, the size of the maximum matching in $G(n,p)$. The exact asymptotics for $M(G(n,p))$ was derived in \cite{Karp1981} and plugging the same, we obtain the upper bound in \eqref{e:regbds}. 

\item Apriori from Theorem \ref{t:lrlp} or \cite[Theorem 1.6]{Erman17}, we obtain that $\reg(I(n,p)) \geq 3$ w.h.p. for $p = \lambda/n$ for $\lambda \in (0,\infty)$ and further w.h.p. $\beta_{i,i+s}(I(n,p)) = 0$ for all $i \in \N$ and $s \geq 4$. However, what we shown above is that $\liminf \frac{\reg(I(n,p))}{n} > 0$ a.s.. By definition, this implies that there exist random sequences $M_n,N_n$\footnote{Here we have used the standard Bachmann-Landau big  O notation.} such that $\beta_{M_n,M_n+N_n}(I(n,p)) \neq 0$ a.s. and $N_n = \Theta(n)$ a.s.. Since $\beta_{i,j}(I(n,p)) \neq 0$ only for $i \leq j \leq 2i$, we have that $M_n \leq M_n + N_n \leq 2M_n$ a.s. and so $M_n = \Omega(n)$ a.s.. Since $E(G(n,p)) = \Theta(n)$ a.s., we also have that $M_n = O(n)$ a.s.. Thus, $M_n = \Theta(n)$ a.s..

\item Our proof technique yields sub-Gaussian concentration and variance bounds as well for regularity, projective dimension and depth.  See \eqref{e:mcdiarmidconc}, Lemma \ref{l:conc_nearlip} and Section \ref{sec:proofreg} for details. 

\item For Krull dimension,  one can obtain the full strong law (i.e.,  Part (2) in the above theorem) for all $\lambda \in (0,\infty)$.  This is because the Krull dimension is same as the independence number of the underlying graph \cite[Theorem 1.33]{van2013} and using  \cite[Theorem 1]{Salez2016} (see also examples and remarks below the theorem therein),  one has strong law for independence number for \ER \, random graphs as well as many other sparse random graphs.  We remark at the end of Section \ref{s:intro} on possibilities for extending to other invariants and other random graph models.
\end{enumerate}
\end{remark} 
We now remark on proof techniques.  By Mcdiarmid's bounded difference inequality and an extension of the same,  we show that it suffices to prove convergence in expectation of $\reg(I(n,p))$ to deduce a.s.  convergence.  To use the bounded difference inequality,  we appeal to the Lipschitz property of regularity and near Lipschitz properties of projective dimension and depth.   For convergence in expectation,  we use the theory of local weak convergence (or Benjamini-Schramm convergence) in the sub-critical regime to relate regularity of \ER \, random graphs to that of its local weak limit,  the Galton-Watson tree with Poisson ($\lambda$) offspring distribution.  The other crucial facts of use are finiteness of the Galton-Watson tree in the sub-critical regime and that the regularity on a tree is equal to that of induced matching number plus one.     

\subsubsection{Unmixedness of random edge ideals}
\label{s:unmix}

We finally investigate unmixedness of edge ideals. We shall show that barring the very extreme values of $p$ and a certain intermediate regime, $I(n,p)$ is not unmixed w.h.p.. We would like to mention that De Loera et al. conjectured that if one generates monomial ideals in the \ER \, way, then with probability one it is not Cohen-Macaulay (  \cite[Conjecure 1]{PetHa2017}) when both the number of variables as well as the degree of the generators go to infinity. This conjecture was shown in \cite[Corollary 1.2]{deLoera2019average} for the case of random monomial ideals when the degree of the generators goes to infinity. Our model is different from theirs in the sense that we take the number of variables going to the infinity.  But motivated by their work it looks pertinent to ask whether similar conclusion is likely in our set up too or not.  We investigate that and get the following result strongly indicating that similar phenomena holds in our set-up as well.
\begin{theorem}
\label{t:unmixed_Inp}
\begin{enumerate}
\item If $p = n^{-\alpha}$ for $\alpha \in (\frac{3}{2},2)$ then $I(n,p)$ is unmixed w.h.p.. 
\item If $p = n^{-\alpha}$ for $\alpha \in [1,\frac{3}{2})$ then $I(n,p)$ is not unmixed w.h.p.. 
\item If $p \in (0,1)$, then $I(n,p)$ is not unmixed w.h.p.. 
\item If $p = 1 - n^{-\alpha}$ for $\alpha \in (0,2)$ and $\alpha^{-1} \notin \{2,3,\ldots\}$, then $I(n,p)$ is not unmixed w.h.p.. 
\item If $p = 1 - n^{-\alpha}$ for $\alpha > 2$ then $I(n,p)$ is unmixed w.h.p..
\end{enumerate}
\end{theorem}
Since unmixedness is equivalent to all minimal vertex covers being of the same size, we shall use structural results about random graphs to show non-unique minimal vertex covers. The main tools in our proof are thresholds for containement of subgraphs in random graphs (see Theorem \ref{t:denrg}), thresholds for maximal cliques and asymptotics of clique numbers. The reason for using cliques is the duality between vertex covers of a graph and cliques in the complement graph. Apart from some boundary cases, we omit $p = n^{-\alpha}, \alpha < 1$ and $p = n^{-\alpha}$ for $\alpha > 2$. The former is omitted because our techniques fail in this regime. In the latter regime, the graph consists only of isolated vertices w.h.p. and hence we have avoided this pathalogical case in our theorem. We have broken down the theorem into multiple regimes as the proof in each regime uses slightly different arguments and also we omit some boundary cases. \\

\textbf{Questions for further research}: The proof of Theorem \ref{t:reginp} via near-Lipschitz property,  local weak convergence and additivity of regularity or projective dimension can also be extended to other random graphs (for example configuration models) with local weak limits and we would expect this approach can be extended further for other invariants as well as in the supercritical regime.  For example,  our proof of law of large numbers applies to any invariant $h(G)$ of the edge ideal that satisfies $|h(G)-h(G\setminus\{v\})| \leq p(\Delta(G))$ where $p$ is any polynomial and $\Delta(G)$ is the maximum degree.  This automatically leads us to a series of questions regarding polynomial Lipschitz bounds for regularity,  projective dimension,  depth and Betti numbers for higher powers of edge ideals as well as similar questions for other graph-related ideals like path ideals,  binomial edge ideals etc..  Determining the limit of the regularity or other algebraic invariants for edge ideals of \ER \, random graphs in the super-critical regime $\lambda > 1$ or even showing its existence remains a challenging problem by itself.  

 \section{Preliminaries}
 \label{sec:prelims}
 
 In this section, we put down the notations and the terminologies that will be used throughout the article. Though most of this is standard, we still define them for a self-contained exposition.  For more details on the algebraic notions, we refer the reader to \cite{BanBayHa2019}, \cite{Beyarslan2015} and \cite{Miller2004}. 
 
\subsection{Combinatorial Preliminaries}
 
 For any $A \subseteq V (G)$, the induced subgraph on $A$ is the maximal subgraph of $G$ whose vertex set is $A$. For any graph $G$, the complement graph, denoted by $G^c$, is the graph whose vertex set is $V(G)$ and the edge set consists of non-edges of $G$, i.e., for $a, b \in V (G), ab \in E(G^c)$ if and only if $ab \notin E(G)$.\\

A cycle of length $n$ in a graph $G$ is a closed walk along its edges, $x_1x_2, x_2x_3, \ldots, x_{n-1}x_n, x_nx_1$, such that $x_i \neq x_j$ for $ i\neq j$. We denote the cycle on $n$ vertices by $C_n$. A chord in the cycle $C_n$ is an edge $x_ix_j$ where $x_j \neq x_{i-1},x_{i+1}$. A graph is said to be {\em chordal} if for any cycle of length greater than or equal to $4$ there is a chord. We say a graph is {\em 4-chordal} if any cycle of length $4$ has a chord. A graph is said to be co-chordal (resp. $4$-cochordal) if the complement of $G$ is chordal ($4$-chordal). The $4$-cochordal graphs are also called {\em gap free} \cite{BanBayHa2019} which we define now. In a graph $G$, we say two disjoint edges $uv$ and $xy$ form an induced gap if $G$ does not have an edge with one endpoint in $\{u,v\}$ and the other in $\{x,y\}$. A graph without an induced gap is called {\em gap-free}. Equivalently, $G$ is gap-free if and only if $G^c$ contains no induced $C_4$.\\

A graph $G$ is said to be {\em locally cochordal} if for every vertex of $G$ the graph obtained by deleting the vertex and all its neighbors is cochordal. A graph $G$ is said to be {\em locally 4 cochordal} if for every vertex of $G$ the graph obtained by deleting the vertex and all its neighbors is 4 cochordal. As one would expect locally cochordal and locally 4 cochordal are implied by cochordal and $4$-cochordal respectively.\\

A {\emph {matching}} in a graph is a collection if edges such that no two of them share any vertex. A matching is called {\emph{induced matching}} if the induced subgraph of G on the vertices belonging to those edges has no other edge. Maximum size of an induced matching is called the induced matching number $\nu(G)$ of the graph. For example, induced matching number of a $5$-cycle is $1$ and that of a $6$-cycle is $2$.\\
 
An {\emph {independent set}} in a graph is a set of vertices such that there is no edge among them. Complement of an independent set is called a {\emph {vertex cover}}. A vertex cover whose no subset is a vertex cover is called a {\emph {minimum vertex cover}}. \\

A forest is a graph without any cycles. A tree is a connected forest. A complete graph (or clique) on $n$ vertices is a graph where for any two vertices there is an edge joining them. It is denoted by $K_n$. A bipartite graph is a graph whose vertices can be split into two groups such that there is no edge between vertices of the same group; only edges are between vertices coming from different groups. It is easy to see that a graph is bipartite if and only if it is without any cycle of odd length.
  
\subsection{Algebraic preliminaries} Let $S = K[x_1, . . . , x_n]$ be the polynomial ring over a field $K$. Let $M$ be a finitely generated $\mathbb{Z}^n$-graded $S$-module. It is known that $M$ can be successively approximated by free modules. Formally speaking, there exists an exact sequence of minimal possible length, called a minimal free resolution of $M$: 
$$\mathbb{F}: \text{ }0  \longrightarrow F_p \overset{d_p} \longrightarrow F_{p-1} \cdots \overset{d_2} \longrightarrow F_1 \overset{d_1} \longrightarrow F_0 \overset{d_0} \longrightarrow M \longrightarrow 0 $$
Here, $F_i = \bigoplus _{\sigma \in \mathbb{Z}^n} S(-\sigma)^{\beta_{i,\sigma}}$ , where $S(-\sigma)$ denotes the free module obtained by shifting the degrees in $S$ by $\sigma$. The numbers $\beta_{i,\sigma}$'s are positive integers and are called the multigraded Betti numbers of $M$. We often identify $\sigma$ with the monomial whose exponent vector is $\sigma$. For example, over $K[x,y]$, we may write $\beta_{i,x^2y}(M)$ instead of $\beta_{i,(2,1)}(M)$.\\

 For every $j \in \mathbb{Z}$, we have  $\beta_{i,j} = \Sigma_ {\sigma : |\sigma|=j} \beta_{i,\sigma}$  is called the $(i,j)$-th standard graded Betti number of $M$. Three very important homological invariants that are related to these numbers are the Castelnuovo-Mumford regularity, or simply regularity, the depth and the projective dimension, denoted by $\reg(M), \text{depth }(M)$ and $\text{pd }(M)$ respectively:
\begin{align*}
\reg(M) & := \max \{|\sigma|-i : \beta_{i,\sigma} \neq 0\}, \\
\text{ depth }M & := \inf \{i : \text{Ext}^i(K,M) \neq 0\}, \\
\text{pd }M & := \max \{i : \mbox{there is a $\sigma$ such that $\beta_{i,\sigma} \neq 0$} \}.
\end{align*}
For an ideal $I$ in $S$ one defines Krull Dimension (denoted by $\dim(S/I)$:

$$\dim(S/I):=\max\{t | \text{There exists prime ideals in }S/I, \text{ } P_0 \subsetneq P_1 \ldots \subsetneq P_t \}$$

If $S$ is viewed as a standard graded $K$-algebra and $M$ is a graded $S$-module, then the
graded Betti numbers of M are also given by $\beta_{i,j}(M) = \text{dim}_K \text{Tor}_i(M,K)_j$, and so we have $\reg(M)=\text{max}\{j-i| \text{Tor}_i(M,K)_j \neq 0\} $ and $\text{pd }M=\text{max}\{i| \text{Tor}_i(M,K) \neq 0 \}$ and $ \text{depth }M = n - \text{pd }M$. Here $\text{Tor}_i(M,K)$ denotes the $i$th homology module of the complex $\mathbb{F} \otimes_S K$, and $\text{Tor}_i(M,K)_j$ denotes its $j$th graded component \cite{Miller2004} \\

 Note that from definition for any graded ideal $I$ in $S$, we have $\beta_{i,j}(I)=\beta_{i+1,j}(\frac{S}{I})$ for all $i\geq 0$.

 We explain this with the following example:\\
 
 Let $M = \frac{\mathbb{Q}[x_1, . . . , x_5]}{(x_1x_2, x_2x_3, x_3x_4, x_4x_5, x_5x_1)}$. Then the minimal free resolution of $M$ is:
 
 $$0  \longrightarrow F_3 \overset{d_3} \longrightarrow F_{2}  \overset{d_2} \longrightarrow F_1 \overset{d_1} \longrightarrow F_0 \overset{d_0} \longrightarrow M \longrightarrow 0 $$

Here: 
$ \beta_{0,1} =1, \beta_{0,\sigma} =0$ otherwise.\\
$\beta_{1, 2}= 5$, and $\beta_{1,j} = 0$ otherwise.\\
$\beta_{2,3}= 5$ and $\beta_{2,j} = 0$ otherwise.\\
$ \beta_{3,  5} = 1$, and $\beta_{3,j} = 0$ otherwise.\\

Note that if we take $M=(x_1x_2, x_2x_3, x_3x_4, x_4x_5, x_5x_1)$ instead then the resolution becomes 

$$0  \longrightarrow F_2 \overset{d_2} \longrightarrow F_{1}  \overset{d_1} \longrightarrow F_0 \overset{d_0} \longrightarrow M \longrightarrow 0 $$

Here: 
$ \beta_{0,2} =5, \beta_{0,j} =0$ otherwise.\\
$\beta_{1, 3}= 5$, and $\beta_{1,j} = 0$ otherwise.\\
$\beta_{2,5}= 1$ and $\beta_{2,j} = 0$ otherwise.\\

An module $M$ is said to have linear presentation if the matrix of $d_0$ have linear or zero entries. A module $M$ is said to have linear resolution if for all $i \geq 1$ the matrices of $d_i$s have linear entries. When $M$ is a homogeneous ideal $I$ generated in degree $d$, it follows from definition that $I$ has linear resolution if and only if $\reg(I)=d$. An ideal $I$ in $S$ is said to be unmixed if all its associated primes are minimal of the same height. We say $I$ (or equivalently $\frac{S}{I}$) is Cohen-Macaulay if the Krull dimension and depth are equal. Cohen-Macaulay ideals are always unmixed and known to have many nice geometric properties. It can be checked that the module $M$ in the example is Cohen-Macaulay.\\

\subsection{Edge Ideals}
\label{s:EI}

 In this section, we identify the combinatorial structures that are related to regularity of edge ideals. Recall the definition of edge ideals from \eqref{d:ei}. Note that regularity of an edge ideal is bounded below by 2, which is the generating degree of an edge ideal. Thus, characterizing combinatorial structures of a graph with regularity two can be considered as a basic question in the subject.The following combinatorial characterization of such graphs is nowadays often referred to as Froberg's characterization. We state this and a characterization of linear presentation now. \\
\begin{theorem}(\cite[Theorem 1.1]{Fr1990})
\label{thm:Froberg}
The edge ideal of a finite simple graph has linear resolution if and only if the graph is co-chordal.
\end{theorem}
\begin{theorem}(\cite[Proposition 1.3]{NePe2013})
\label{thm:Froberg1}
The edge ideal of a finite simple graph has linear presentation if and only if the graph is $4$-co-chordal.
\end{theorem}
For any ideal $I$ and any element $a$ in $S$ the colon ideal, $I:a$ is the ideal $(b \in S| ab \in I)$. We note that for any graph $G$ and any vertex $x$, the ideal $I(G):x$ is the edge ideal of the graph obtained by deleting $x$ and its neighbors along with any resultant isolated vertices. More specifically, $I(G):x = I(G \setminus Nb[x]) + (y | y \in Nb(x) )$,here $Nb(x)$ denotes the set of neighbors of $x$ and $Nb[x]= Nb(x) \cup \{x\}$. As a convention, we say that a graph $G$ or its edge ideal $I(G)$ has local linear resolution if for every vertex $x$ the ideal $(I(G):x)$ has linear resolution. Similarly we say that a graph $G$ or its edge ideal $I(G)$ has local linear presentation if for every vertex $x$ the ideal $(I(G):x)$ has linear presentation. \\

From the definitions of local linear resolution and presentation along with Theorems \ref{thm:Froberg} and \ref{thm:Froberg1}, we obtain the following obvious characterization.
\begin{corollary}
\label{c:loclinrescondn}
The edge ideal of a finite simple graph has local linear resolution if and only if the graph is locally co-chordal. Similarly, the edge ideal of a finite simple graph has local linear presentation if and only if the graph is locally $4$-co-chordal.
\end{corollary}  

A set of vertices $S$ forms a {\em vertex cover} if every edge in the graph is incident on at least one vertex from the set $S$. It is possible that both the vertices of an edge are in a given vertex cover. A vertex cover is called {\em minimal vertex cover} if no proper subset is a vertex cover. One observes that the Krull Dimension of $S/I(G)$ is the difference between the number of vertices in $G$ and the minimum size of vertex cover of $G$ \cite{deLoera2019average}.  Note that this is the same as the maximum size of an independent set of vertices of $G$ (an independence set is the complement of a vertex cover in a graph and of interest as a concept in its own right). 
\begin{theorem}( \cite[Section 3]{van2013})
\label{t:unmixed}
A graph is unmixed if all minimal vertex covers have same size.
\end{theorem}

We note that isolated vetrices do not affect any of the notions we study here. The edge ideal of a graph remains unchanged if one adds a few isolated vertices. So does the properties of chordality, cochordality, 4-chordality, 4-cochordality etc.


\subsection{\ER \, random graphs}
\label{s:probprelims}
Recall the definition of \ER \, random graph from \eqref{d:erg}.
We collect now results regarding random graphs that we need in this paper. The reader may refer to these results when necessary. For more on \ER \, random graphs, please refer to \cite{Frieze2016}.\\

We denote density of a graph $G$ with $n$ vertices and $m$ edges to be  $d(G) = m/n$. A graph is said to be {\em strictly balanced} if the density of the graph itself is strictly greater than the density of any of its subgraphs. It is trivial to see that a complete graph on $k$ vertices is strictly balanced with density $\frac{k-1}{2}$ and a tree on $k$ vertices is strictly balanced with density $\frac{k-1}{k}$. It is also trivial to see that cycles and edges are also strictly balanced subgraphs.  By $N_H(G)$, we denote the number of copies of $H$ in $G$. We will use $\stackrel{D}{\rightarrow}$ to denote convergence in distribution i.e., we say that a sequence of random variables $X_n \stackrel{D}{\rightarrow} X$ if $F_{X_n}(x) \to F_X(x)$ for all $x$ at which $F_X$ is continuous where $F_{X_n}(\cdot)$ and $F_{X}(\cdot)$ denote the CDFs (cumulative distribution functions) of the random variables $X_n$ and $X$ respectively.
\begin{theorem}(\cite[Theorems 5.3 and 5.4]{Frieze2016})
\label{t:denrg}
Let $H$ be a strictly balanced subgraph on $k$ vertices and of density $d := d(H)$. Then for $p = n^{-\alpha}$, $H$ is a subgraph of $G(n,p)$ w.h.p. if $\alpha < \frac{1}{d}$ and if $\alpha > \frac{1}{d}$ then $H$ is not a subgraph of $G(n,p)$ w.h.p. 

Further, if $np^d \to \lambda$ then $N_H(G(n,p)) \stackrel{D}{\rightarrow} Z_{\frac{\lambda^k}{aut(H)}}$ where $aut(H)$ is the number of automorphisms of $H$ and $Z_a$ denotes the Poisson random variable with mean $a$ for $a \in [0,\infty)$. 
\end{theorem}

\paragraph{Sparse \ER \, random graphs :} A specific choice of $p$ that shall play an important role in our analysis is $p := p(n) = \min  \{\lambda/n,1\}$ for $\lambda \in (0,\infty)$. In this case, $\EXP{e(n,p)} = \binom{n}{2}p \sim \frac{1}{2}\lambda n$ as $n \to \infty$ where $e(n,p)$ is the cardinality of the edge-set $E(n,p)$. The random graph is called sparse in this regime as the average degree $n^{-1}e(n,p)$ is bounded.  A particular feature of $G(n,p)$ in the sparse regime is that it has very few cycles and it resembles a forest (i.e., every component is a tree). We shall very much rely upon this fact in the proof of Theorem \ref{t:reginp} and also in determining the critical windows for linear resolution and linear presentation in Section \ref{sec:lplr}. A key result towards this is to characterize the joint distribution of cycles in $G(n,p)$ in the sparse regime.
\begin{theorem}(\cite[Theorem 2.15]{Bordenave2016})
\label{t:multPoisson}
For $l \geq 3$, let $C_l$ denote {\em the cycle graph} on $l$ vertices. Let $np \to \lambda \in (0,\infty)$. Then for any $k \geq 3$ and $(a_3,\ldots,a_k) \in \{0,1\}^k$, we have that 
$$ \sum_{i=3}^k a_i N_{C_i}(G(n,p)) \stackrel{D}{\rightarrow} Z_{\sum_{i=3}^k \frac{a_i\lambda^i}{2i}}.$$
\end{theorem}

\paragraph{Galton-Watson tree :} The intuition that $G(n,p)$ "locally resembles" a forest is made precise via the theory of local weak convergence (or Benjamini-Schramm convergence). It is shown that the \ER \, random graph "locally resembles" a Galton-Watson tree $GW(\lambda)$ with Poisson($\lambda$) offspring distribution. We shall describe this tree first informally now. Let $\emptyset$ denote the root vertex. This vertex has a Poisson ($\lambda$) number of neighbours. Each of these vertices have a further Poisson  ($\lambda$) number of new neighbours and independent of each other. Further, none of the new neighbours are common. These new neighbours also have a Poisson  ($\lambda$) number of newer neighbours and so on. \\

We now define the Galton-Watson tree formally but in a terse manner (see \cite[Section 3.4]{Bordenave2016} for more details). First, we set-up some notation for infinite trees. Define $\N^f := \cup_{k \geq 0}\N^k$ with $\N^0 := \{\emptyset\}$. For $\bi = (i_1,\ldots,i_k) \in \N^f$, we call $(i_1,\ldots,i_{k-1})$ the {\em ancestor} of $\N$. Given a non-negative integer-valued sequence $\{N_{\bi}\}_{\bi \in \N^f}$, we set the vertex set of an infinite tree as
$$V = \{\emptyset \} \cup \{ \bi = (i_1,\ldots,i_k) : \forall 1 \leq l \leq k, 1 \leq i_l \leq N_{(i_1,\ldots,i_{l-1})} \}.$$
We call $(\bi,1),\ldots,(\bi,N_{\bi})$ to be {\em children} or {\em off-springs} of $\bi$. The infinite rooted tree $T$ is defined as the tree with vertex set $V$, $\emptyset$ as the root and undirected edges between vertices and their ancestors. By definition for all $\bi \in V$ with $i \neq \emptyset$, the ancestor of $i$ is also in $V$ and further, the degree of such a vertex $\bi$ is $N_{\bi} + 1$. The degree of root $\emptyset$ is $N_{\emptyset}$. \\

The Galton-Watson tree with Poisson($\lambda$) offspring distribution ( denoted by $GW(\lambda)$) is simply the rooted tree with the sequence $\{N_{\bi}\}_{\bi \in \N^f}$ chosen to be i.i.d. (independent and identically distributed) Poisson($\lambda$) random variables. \\

We refer the reader interested in more details about local weak convergence to \cite{Bordenave2016,Hofstad20,Aldous04}. We shall exploit the fact that limiting  structure of  the \ER \, random graph is a forest to prove certain asymptotics for regularity of the edge ideal of \ER \, random graphs. Though we will be using some results derived from local weak convergence theory, we avoid introducing this here as it is not necessary to understand our results or proofs. \\

\paragraph{A Concentration inequality for Lipschitz functionals :} We shall now state a concentration inequality for Lipschitz graph functionals that we need. Let $\psi$ be a {\em graph functional} i.e., $\psi : \mathcal{G} \to \R$ where $\mathcal{G}$ is the collection of all locally-finite graphs and $\psi(G_1) = \psi(G_2)$ if $G_1$ and $G_2$ are isomorphic. Suppose that $\psi$ is {\em Lipschitz} under vertex-addition i.e., there exists $M > 0$ such that for any graph $G$ and $v \in G$, we have that
\begin{equation}
\label{e:lip}
| \psi(G) - \psi(G \setminus v)| \leq M
\end{equation}
Then, for graphs $G_1$ and $G_2$ with same vertex set and $G_1 - v = G_2 - v$, we have that
$$ |\psi(G_1) - \psi(G_2)| \leq | \psi(G_1) - \psi(G_1 \setminus v)| + | \psi(G_2) - \psi(G_2 \setminus v)| \leq 2M.$$
Now, by using Mcdiarmid's inequality (see \cite[Lemma 21.16]{Frieze2016}) as in \cite[Theorem 7.8]{Frieze2016}, we can derive that for $t > 0$,
\begin{equation}
\label{e:mcdiarmidconc}
 \pr{ |\psi(G(n,p)) - \EXP{\psi(G(n,p))}| \geq t} \leq 2\exp\{-\frac{t^2}{4nM^2}\}.
\end{equation}
As a consequence we have that $\limsup n^{-1}\VAR{\psi(G(n,p))} < \infty$ and further by Chebyshev's inequality, we can derive that
\begin{equation}
\label{e:conc_lipfunctionals}
\lim_{n \to \infty} \frac{\psi(G(n,p)) - \EXP{\psi(G(n,p))}}{n} = 0, \, \, \mbox{a.s..}
\end{equation}

However, the above result is not sufficient to conclude a strong law for Lipschitz functionals as we do not have convergence of expectations. We need further assumptions (for example, see \cite[Theorem 1]{Salez2016}) for the same.  But we shall state two strong laws that we have used in Remark \ref{r:ermanyangcomp2}. Let $\beta_0(G)$ and $\hat{\beta}_0(G)$ be the number of connected components and number of non-trivial connected components (i.e., size is at least $2$) respectively.
Denoting the number of isolated vertices by $N_0(G)$, observe that
$$ \beta_0(G) = N_0(G) + \hat{\beta}_0(G).$$
Further, it is easy to check that all three functionals are Lipschitz but under further assumptions, a strong law for these functionals was proven for $G(n,\lambda/n)$ (see examples in \cite[Theorem 1]{Salez2016} and also remarks below Theorem 1 therein) : There exist constants $\beta_{\infty}(\lambda), \hat{\beta}_{\infty}(\lambda) \in (0,\infty)$ such that
\begin{equation}
\label{e:betainf}
\beta_{\infty}(\lambda) = \lim_{n \to \infty} n^{-1} \beta_0(G(n,p)) \, \, \mbox{a.s.},
\end{equation}
and
\begin{equation}
\label{e:betahinf}
\hat{\beta}_{\infty}(\lambda) = \lim_{n \to \infty} n^{-1} \hat{\beta}_0(G(n,p)) \, \, \mbox{a.s.}.
\end{equation}
The above convergences also hold in expectation. Further, since $\EXP{N_0(G(n,p)} = n(1-p)^{n-1}$, we have that 
$$ \lim_{n \to \infty} n^{-1} \EXP{N_0(G(n,p)} = e^{-\lambda}.$$
and so $\hat{\beta}_{\infty}(\lambda)  = \beta_{\infty}(\lambda) - e^{-\lambda}$ by the earlier relation between $\beta_0, \hat{\beta}_0$ and $N_0$. \\

We will now present an extension of \eqref{e:mcdiarmidconc} that shall be useful in extending the scope of our applications. 
\begin{lemma}
\label{l:conc_nearlip}
Let $p = \frac{\lambda}{n}$ for $\lambda > 0$. Let $\Psi$ be a graph functional that satisfies \eqref{e:lip} with $\Delta(G)$ instead of $M$ i.e., for all $v \in G$, 
$$ | \psi(G) - \psi(G\setminus v)| \leq \Delta(G) := \max_{u \in G}deg(u).$$
Then, we have that for $n \in \N$ and any $t > 0, M \in \N$,
\begin{align*}
\pr{|\psi(G(n,p)) - \EXP{\psi(G(n,p))}| > t} \leq 2\exp\{-\frac{t^2}{4nM^2}\} + 2n^2 \left(\frac{\lambda e}{M} \right)^M.
\end{align*}
Further,  
$$\lim_{n \to \infty} \frac{\psi(G(n,p)) - \EXP{\psi(G(n,p))}}{n} = 0, \, \, \mbox{a.s..}$$
\end{lemma}
\begin{proof}
We define random elements $Y_1,\ldots,Y_{n-1}$ as follows : 
$$ Y_i := \{ (i,j) : i < j, (i,j) \in E(n,p) \}.$$
By the definition of the \ER \, random graph, $Y_1,\ldots,Y_n$ are independent random elements. Observe that we can write $G(n,p) = g(Y_1,\ldots,Y_{n-1})$ for a measurable function $g$ and so $\psi(G(n,p)) = f(Y_1,\ldots,Y_{n-1})$ for a suitably measurable function $f$. Then using the fact that \\ $\EXP{f(Y_1,\ldots,Y_{n-1}) \mid Y_1,\ldots,Y_k}, k = 0,\ldots,n$ is a martingale sequence and \cite[Theorem 8.4]{Chung2006}, we have that for any $t, M > 0$,
\begin{align}
& \pr{|\psi(G(n,p)) - \EXP{\psi(G(n,p))}| > t} \no \\
& \leq \quad \pr{|f(Y_1,\ldots,Y_{n-1}) - \EXP{f(Y_1,\ldots,Y_{n-1})}| > t}  \leq 2\exp\{-\frac{t^2}{4nM^2}\} \no \\
\label{e:conc_nearlip} & \quad + \sum_{k=1}^{n-1}\pr{|\EXP{f(Y_1,\ldots,Y_{n-1}) \mid Y_1,\ldots,Y_k} - \EXP{f(Y_1,\ldots,Y_{n-1}) \mid Y_1,\ldots,Y_{k-1}}| \geq 2M}.
\end{align}
Denoting by $Y'_1,\ldots,Y'_{n-1}$ independent copies of $Y_1,\ldots,Y_{n-1}$, we have that
\begin{align*}
& |\EXP{f(Y_1,\ldots,Y_{n-1}) \mid Y_1,\ldots,Y_k} - \EXP{f(Y_1,\ldots,Y_{n-1}) \mid Y_1,\ldots,Y_{k-1}}| \\
& \quad \leq \EXP{|f(Y_1,\ldots,Y_{n-1}) - f(Y_1,\ldots,Y_{k-1},Y'_k,\ldots,Y_{n-1})| \mid Y_1,\ldots,Y_{k-1}}.
\end{align*}
Let $G'(n,p) = g(Y_1,\ldots,Y_{k-1},Y'_k,\ldots,Y_{n-1})$. Observe that $G(n,p) \setminus \{k\} - G(n,p) \setminus \{k\}$ and so as reasoned for Lipschitz under vertex-addition, 
$$ |f(Y_1,\ldots,Y_{n-1}) - f(Y_1,\ldots,Y_{k-1},Y'_k,\ldots,Y_{n-1})| = |\psi(G(n,p)) - \psi(G'(n,p))| \leq \Delta(G(n,p)) + \Delta(G'(n,p)).$$
Let $D(i)$ denote the degree of the vertex $i$ in $G(n,p)$ and choose $M \in \N$. Observe that $D(i), i =1,\ldots,n$ are identically distributed Binomial($n-1,p$) random variables. Using that $G(n,p) \stackrel{d}{=} G'(n,p)$ and substituting the above bound in \eqref{e:conc_nearlip}, we obtain that for all $k = 1,\ldots,n-1$,
\begin{align*}
& \pr{|\EXP{f(Y_1,\ldots,Y_{n-1}) \mid Y_1,\ldots,Y_k} - \EXP{f(Y_1,\ldots,Y_{n-1}) \mid Y_1,\ldots,Y_{k-1}}| \geq 2M} \\
& \quad \leq \pr{\Delta(G(n,p)) \geq M} +  \pr{\Delta(G'(n,p)) \geq M} \\
& \quad = 2\pr{\Delta(G(n,p)) \geq M} \leq 2 \sum_{i=1}^n \pr{D(i) \geq M} =2n \pr{D(1) \geq M} \\
& \quad \leq 2n\binom{n-1}{M}\left( \frac{\lambda}{n} \right)^M \leq 2n \left(\frac{\lambda e}{M} \right)^M
\end{align*}
This proves the concentration inequality. 

Now,  by choosing $M = \lfloor n^{\epsilon} \rfloor$ for $\epsilon > 0$ small enough,  we have that for any $t > 0$, 
$$ \sum_{n \geq 1} \pr{|\psi(G(n,p)) - \EXP{\psi(G(n,p))}| > tn} < \infty,$$
and hence the a.s. \, convergence follows from Borel-Cantelli Lemma. 
\end{proof}

\section{Proofs}
\label{sec:EIRG}
Recall that $G(n,p)$ is the \ER \, random graph on $n$ vertices with edge probability $p$ and $I(n,p) = I(G(n,p))$, the edge-ideal generated by $G(n,p)$. 
\subsection{Proofs of Theorems \ref{t:lrlp} and \ref{t:lrlplambda}.}
\label{sec:proofslrlp}
 
\begin{proof}[Proof of Theorem \ref{t:lrlp}]
By Theorems \ref{thm:Froberg} and \ref{thm:Froberg1} and that $G(n,p)^c \stackrel{d}{=} G(n,1-p)$, it suffices to show that
\begin{equation}
\label{e:Gnp4chordal}
\lim_{n \to \infty} \pr{\mbox{$G(n,p)$ is $4$-co-chordal}} = 
\lim_{n \to \infty} \pr{\mbox{$G(n,1-p)$ is $4$-chordal}} =
\begin{cases} 
1 \, \, \, \, \mbox{if \, $(n(1-p))^4p^2 \to 0$}, \\
0 \, \, \, \, \mbox{if \, $(n(1-p))^4p^2 \to \infty$},\\
\end{cases}
\end{equation}
 and also for any sequence $p := p(n)$ such that $n(1-p) \to b \in \{0,\infty\}$, we have that
\begin{equation}
\label{e:Gnpchordal}
\lim_{n \to \infty} \pr{\mbox{$G(n,1-p)$ is not chordal but $4$-chordal}} = 0.
\end{equation}
Let $C'_k(n,p)$ denote the number of cycles in $G(n,p)$ of length $k$ without a chord. Note the following
\begin{align}
\{\mbox{$G(n,1-p)$ is not $4$-chordal}\} & = \{C'_4(n,1-p) \geq 1\} \no \\
\{\mbox{$G(n,1-p)$ is not chordal}\} & = \cup_{k \geq 4} \{C'_k(n,1-p) \geq 1\} \no 
\end{align}
Observe that
$$C'_k(n,1-p) := \frac{1}{2k}\sum^{\neq}_{i_1,\ldots,i_k} \1[\mbox{$i_1,\ldots,i_k$ form a cycle in $G(n,1-p)$ without a chord}],$$
where $\sum^{\neq}$ denotes that the sum is over distinct indices and $2k$ is to account for the fact that there are $2k$ ordered $k$-tuples giving rise to the same cycle. Now, by linearity of expectations and also that the indicator random variables above are identically distributed, we have that 
\begin{align}
\EXP{C'_k(n,1-p)} &= \frac{(k-1)!}{2} \binom{n}{k} \pr{\mbox{$1,\ldots,k$ form a cycle in $G(n,1-p)$ without a chord}} \no \\
& = \frac{(k-1)!}{2} \binom{n}{k}(1-p)^kp^{\binom{k}{2}-k}. \label{e:expck}
\end{align}
Now using Markov's inequality, we trivially obtain that
$$ \pr{\mbox{$G(n,1-p)$ is not $4$-chordal}} \leq \EXP{C'_4(n,1-p)} = 3 \binom{n}{4} (1-p)^4 p^2 \to 0$$
as $n \to \infty$ if $(n(1-p))^4p^2 \to 0$. This completes the first part of \eqref{e:Gnp4chordal}.\\

We shall now prove the second part of \eqref{e:Gnp4chordal} via the second moment method. Assume that $(n(1-p))^4p^2 \to \lambda \in (0,\infty]$. From the standard second moment bound, we have that
\begin{equation}
\label{e.secmom}
 \pr{\mbox{$G(n,1-p)$ is not $4$-chordal}} = \pr{C'_4(n,1-p) \geq 1} \geq \frac{\EXP{C'_4(n,1-p)}^2}{\EXP{C'_4(n,1-p)^2}}.
 \end{equation}
We shall now upper bound $\EXP{C'_4(n,1-p)^2}$. In a similar manner as we derived an expression for $\EXP{C'_k(n,1-p)}$, we obtain that
\begin{align*}
\EXP{C'_4(n,1-p)^2} & = \frac{1}{64}\sum^{\neq}_{i_1,i_2,i_3,i_4}\sum^{\neq}_{j_1,j_2,j_3,j_4} \\
& \pr{\mbox{Both $i_1,i_2,i_3,i_4$ and $j_1,j_2,j_3,j_4$ form cycles in $G(n,1-p)$ without a chord}} .
\end{align*}
Writing for $j = 0,1,\ldots,4$,
\begin{align*}
T_j & = \frac{1}{64}\sum^{\neq}_{I = \{i_1,i_2,i_3,i_4\}}\sum^{\neq}_{J= \{j_1,j_2,j_3,j_4\}, |I \cap J| = j} \\
& \pr{\mbox{Both $i_1,i_2,i_3,i_4$ and $j_1,j_2,j_3,j_4$ form cycles in $G(n,1-p)$ without a chord}},
\end{align*}
we have that
$$  \EXP{C'_4(n,1-p)^2}  = \sum_{j=0}^4T_j.$$
By \eqref{e.secmom}, the proof of second part of \eqref{e:Gnp4chordal} is complete if we show that for $c = \infty$
\begin{equation}
\label{e:Tjbds}
\lim_{n \to \infty}\frac{T_0}{\EXP{C'_4(n,1-p)}^2} = 1, \lim_{n \to \infty}\frac{T_j}{\EXP{C'_4(n,1-p)}^2} = 0, j = 1,2,3,4.
\end{equation}
For the first part of \eqref{e:Tjbds}, by proceeding as in the derivation of the expectation, we have that
$$T_0 = 9\binom{n}{4}\binom{n-4}{4}(1-p)^8p^4.$$
Similarly, we obtain for $T_1$ that
$$T_1 = 9\binom{n}{4}\binom{n-4}{3}(1-p)^8p^4 \leq \frac{4}{n-7}T_0.$$
However for $T_2$ we have to consider two possibilities : The two common vertices could be sharing an edge or they might not be sharing an edge. 
$$T_2 = C_1\binom{n}{4}\binom{n-4}{2}(1-p)^7p^4 + C_2\binom{n}{4}\binom{n-4}{2}(1-p)^8p^3,$$
for some constants $C_1,C_2$. The computations for $T_3,T_4$ proceed along similar lines and thus we obtain 
$$T_3 \leq Cn^5(1-p)^6p^3 = C\frac{T_0}{n^3(1-p)^2p},$$
and
$$T_4  = 3\binom{n}{4}(1-p)^4p^2 = \EXP{C'_4(n,1-p)},$$
where $C$ is some constant.  If $\lambda = \infty$, the above bounds along with \eqref{e:expck} suffice to show that \eqref{e:Tjbds} holds and hence we obtain the second part of \eqref{e:Gnp4chordal} as well. \\

Now to prove \eqref{e:Gnpchordal}, we only need to consider the case that $(n(1-p))^4p^2 \to \lambda \in [0,\infty)$ because if $(n(1-p))^4p^2 \to \infty$ then by the second part of \eqref{e:Gnp4chordal}, \eqref{e:Gnpchordal} holds trivially. Now, we derive by Markov's inequality and \eqref{e:expck} that
\begin{align}
\pr{\mbox{$G(n,1-p)$ is not chordal but is $4$-chordal}}  & = \pr{\cup_{k \geq 5} \{C'_k(n,1-p) \geq 1\} \cap \{C'_4(n,1-p) = 0\}} \no \\
& \leq   \pr{\cup_{k \geq 5} \{C'_k(n,1-p) \geq 1\} } \no \\
& \leq \sum_{k \geq 5} \EXP{C'_k(n,1-p)} \no \\
& = \sum_{k \geq 5} \frac{(k-1)!}{2} \binom{n}{k}(1-p)^kp^{\binom{k}{2}-k} \no \\
& \leq (n(1-p))^4p^2 \sum_{k \geq 1} \frac{(n(1-p))^k p^{\binom{k+4}{2}-k-6}}{2(k+4)} \no \\
& =  (n(1-p))^4p^2 \sum_{k \geq 1} \frac{(n(1-p)p^{(k+5)/2})^k}{2(k+4)} \label{e:sumeck}
\end{align}
Let $\lambda = 0$. Then we can easily see that for $n$ large, $n(1-p)p^{5/2} < 1$. Now, let $\lambda >0$ and since $\lim_{n \to \infty} n(1-p) \notin (0,\infty)$, we have that $p \to 0$ and $n^4p^2 \to \lambda$. Thus again for $n$ large, $n(1-p)p^{5/2} < 1$. Hence in all the cases for large $n$, we have from the above derivation that
$$ \pr{\mbox{$G(n,1-p)$ is not chordal but is $4$-chordal}}  \leq (n(1-p))^4p^2 \left( \frac{n(1-p)p^{5/2}}{1 - n(1-p)p^{5/2}} \right) \to 0,$$
as $n \to \infty$ since $n(1-p)p^{5/2} \to 0$.
\end{proof}
\begin{proof}[Proof of Theorem \ref{t:lrlplambda}]
Assume that $(n(1-p))^4p^2 \to \lambda \in (0,\infty)$ and $p \to 0$. Then we have that $n\sqrt{p} \to \lambda^{1/4}$. Denoting the number of edges in $G(n,p)$ by $E(n,p)$, we have from Theorem \ref{t:denrg} that $E(n,p) \stackrel{D}{\rightarrow} Z_{\frac{\sqrt{\lambda}}{2}}$ where $Z_a$ is the Poisson random variable with mean $a$ for $a \in [0,\infty)$. Suppose we denote the event $\{ \mbox{$G(n,p)$ contains a connected subgraph of more than two vertices}\}$ by $A(n,p)$. Then, from Theorem \ref{t:denrg}, we derive that $\pr{A(n,p)} \to 0$ as $n \to \infty$. Hence, w.h.p. $G(n,p)$ consists of $E(n,p)$ many disjoint edges. Observe that if the graph consists of disjoint edges then it is $4$-cochordal iff there is at most $1$ edge i.e.,
$$ \{ \mbox{$G(n,p)$ is $4$-cochordal} \} \cap A(n,p)^c = \{ E(n,p) < 2 \} \cap A(n,p)^c.$$
Thus combining the above identity with the distributional convergence of $E(n,p)$, we can derive that
\begin{align*}
 \pr{\mbox{$G(n,p)$ is $4$-cochordal}} & = \pr{E(n,p) < 2} - \pr{\{E(n,p) < 2\} \cap A(n,p)} \\
 & \, \, \, \, \, + \pr{\{\mbox{$G(n,p)$ is $4$-cochordal}\} \cap A(n,p)} \\
 &\to \pr{Z_{\frac{\sqrt{\lambda}}{2}} < 2}  = e^{-\frac{\sqrt{\lambda}}{2}}(1 + \frac{\sqrt{\lambda}}{2}), \, \mbox{as $n \to \infty$}.
\end{align*}
Now from Theorem \ref{thm:Froberg}, Theorem \ref{thm:Froberg1} and \eqref{e:Gnpchordal}, the proof of \eqref{e:cvgprlrlp1} is complete. \\

Now let $(n(1-p))^4p^2 \to \lambda \in (0,\infty)$ and $p \to 1$. Define $\hat{C}_k(n,1-p)$ to be the number of copies of $k$-cycle $C_k$ in $G(n,1-p)$ i.e.,
$$\hat{C}_k(n,1-p) := \frac{1}{2k}\sum^{\neq}_{i_1,\ldots,i_k} \1[\mbox{$i_1,\ldots,i_k$ form a cycle in $G(n,1-p)$}].$$
Though we have used $N_{C_k}$ instead of $\hat{C}_k$ before, we shall use $\hat{C}_k$ for convenience of notation. Trivially we have that $C'_k(n,1-p) \leq \hat{C}_k(n,1-p)$ and further since $p \to 1$, from \eqref{e:expck}, we obtain that for all $ k \geq 4$,
\begin{equation}
\label{e:ckc'k}
\EXP{\hat{C}_k(n,1-p) - C'_k(n,1-p)} = \frac{(k-1)!}{2}\binom{n}{k}(1-p)^k(1 - p^{\binom{k}{2}-k}) \to 0.
\end{equation}
Then by Theorem \ref{t:denrg} and Slutsky's lemma, we have that $C'_k(n,1-p) \stackrel{D}{\to} Z_{\lambda/8}$. Thus, we obtain that
$$ \pr{\mbox{$G(n,p)$ is $4$-cochordal}} = \pr{C'_k(n,1-p) = 0} \to e^{-\lambda/8}$$
and the proof of the first statement in \eqref{e:cvgprlrlp2} is complete by Theorem \ref{thm:Froberg1}. 

Observe that
\begin{equation}
\label{e:cochordsumck}
\pr{\mbox{$G(n,p)$ is cochordal}} = \pr{\sum_{k=4}^\infty C'_k(n,1-p) = 0} 
\end{equation}
Now from Theorem \ref{t:multPoisson}, \eqref{e:ckc'k} and Slutsky's lemma, we have that for any $m \geq 4$
$$ \sum_{k=4}^m C'_k(n,1-p) \stackrel{D}{\to} Z_{\sum_{k = 4}^m \frac{\lambda^{k/4}}{2k}}.$$
Thus, we can derive that for any $m \geq 1$, 
\begin{equation}
\label{e:ck0cvg}
\lim_{n \to \infty}\pr{\sum_{k=4}^m C'_k(n,1-p) = 0} = e^{-\sum_{k = 4}^m \frac{\lambda^{k/4}}{2k}}.
\end{equation}
Now, assume that  $\lambda < 1$. Using the fact that the events $\{\sum_{k=4}^m C'_k(n,1-p) = 0\}$ are decreasing in $m$,  Markov's inequality, \eqref{e:expck} and following the derivation as in \eqref{e:sumeck}, we have that for any $m \geq 1$, 
\begin{align*}
0 & \leq \pr{\sum_{k=4}^m C'_k(n,1-p) = 0} - \pr{\sum_{k=4}^\infty C'_k(n,1-p) = 0} \\ 
& = \pr{\sum_{k > m} C'_k(n,1-p) \geq 1} \leq \sum_{k > m}\EXP{C'_k(n,1-p)} \\
& \leq (n(1-p))^mp^{(1+ m^2 +5m)/2} \sum_{k \geq 1} \frac{(n(1-p)p^{(k+5)/2})^k}{2(k+4)}
\end{align*}
Since $n(1-p) \to \lambda^{1/4}$ and $\lambda < 1$, we can make $n(1-p)^m$ arbitrarily small by choosing $m$ large and also the final sum in the above derivation is finite. Thus for any $\epsilon >0$, we can find $m$ large such that 
$$ \limsup_{n \to \infty} \left( \pr{\sum_{k=4}^m C'_k(n,1-p) = 0} - \pr{\sum_{k=4}^\infty C'_k(n,1-p) = 0} \right) \leq \epsilon.$$
Now combining the above bound with \eqref{e:ck0cvg}, we obtain that for any $\epsilon > 0$, we can choose $m$ large such that 
$$ \limsup_{n \to \infty} \left( e^{-\sum_{k = 4}^m \frac{\lambda^{k/4}}{2k}} - \pr{\sum_{k=4}^\infty C'_k(n,1-p) = 0} \right) \leq \epsilon.$$
This along with \eqref{e:cochordsumck} and Theorem \ref{thm:Froberg} completes the proof of the second statement in \eqref{e:cvgprlrlp2} for $\lambda < 1$.

Suppose that $\lambda \geq 1$. Then, for any $m \geq 1$ we have by the decreasing property and \eqref{e:ck0cvg} that
$$
\pr{\sum_{k=4}^\infty C'_k(n,1-p) = 0} \leq \pr{\sum_{k=4}^m C'_k(n,1-p) = 0}  \to e^{-\sum_{k = 4}^m \frac{\lambda^{k/4}}{2k}},$$
as $n \to \infty$. Now if $\lambda \geq 1$, the last term can be made arbitrarily small by choosing $m$ large and this completes the proof of  \eqref{e:cvgprlrlp2} for $\lambda \geq 1$ in the same manner as above.
\end{proof}
Recall the notions of local linear presentation and resolution defined in Section \ref{s:EI}. In light of the above results, it begs the question whether local linear resolution implies linear presentation for random graphs or equivalently linear resolution. We show that this need not be the case and explicitly give a parameter regime (albeit a very narrow one) where this will not hold.
\begin{proposition}
\label{p:lcchord}
If $n(1-p)^{8/5} \to 0$ or $n\sqrt{p} \to 0$ then
$$ \lim_{n \to \infty} \pr{\mbox{$I(n,p)$ has local linear resolution}} =  \lim_{n \to \infty} \pr{\mbox{$I(n,p)$ has local linear presentation}} = 1.$$
\end{proposition}
\begin{remark}
\label{r:lcchordcchord}
We shall first discuss the above proposition in relation to linear resolution. From Theorem \ref{t:lrlp}, we have that if $n(1-p) \to \infty$ and $n(1-p)^{8/5} \to 0$ then 
$$ \pr{\mbox{$I(n,p)$ does not have linear presentation but local linear resolution}} \to 1.$$
Also trivially from Theorem \ref{t:lrlp}, we have that the above probability converges to $0$ if $n(1-p)\sqrt{p} \to 0$ and so this leaves open only the case $n^5(1-p)^8p^2 \to \infty$. In this case, we conjecture that the probability of $I(n,p)$ having local linear resolution or presentation converges to $0$. Possibly, this can be proven via second-moment method as for linear resolution and presentation but we do not pursue this in the article. 
\end{remark}

\begin{proof}[Proof of Proposition \ref{p:lcchord}] 
Due to Corollary \ref{c:loclinrescondn}, we need to only prove the statements in the Proposition for $G(n,p)$ being locally cochordal and locally $4$-cochordal respectively. This is similar to the proof of Theorem \ref{t:lrlp}. \\

We define $C_k^*(n,1-p)$ denote the number of $k$-cycles without a chord in $G(n,1-p)$ such that there is a vertex in the graph that is not connected in $G(n,p)$ to any of the vertices in the $k$-cycle. Again observe that
\begin{align*}
\{\mbox{$G(n,p)$ is locally $4$-cochordal}\} &=  \{C_4^*(n,1-p) = 0 \} , \\ 
\{\mbox{$G(n,p)$ is locally cochordal}\} &=  \cap_{k=4}^{\infty}\{C_k^*(n,1-p) = 0 \}.
\end{align*}
Now we proceed as in the proof of Theorem \ref{t:lrlp} by computing expectations of $C_k^*$. 
$$ \EXP{C_k^*(n,1-p)} = \frac{(k-1)!}{2}\binom{n}{k}(1-p)^kp^{\binom{k}{2}-k}(1 - (1 - (1-p)^k)^{n-k}).$$
We justify the additional term as follows. Consider the event that every other vertex is connected to at least one of the $k$ vertices in the cycle. The probability a given vertex is connected to at least one of the $k$ vertices in the cycle is easily seen to be $1 - (1-p)^k$ and since these events are indepdendent for different vertices, we derive that the probability every other vertex is connected to at least one of the $k$ vertices in the cycle is $(1 - (1-p)^k)^{n-k}$.  So the complementary probability of at least one vertex not being connected in $G(n,p)$ to any of the vertices in the $k$-cycle is $(1 - (1 - (1-p)^k)^{n-k})$. \\

Under our assumption that $n(1-p)^{8/5} \to 0$, we have that $n(1-p)^4 \to 0$ as well and so we can derive that 
\begin{align*}
\EXP{C_4^*(n,1-p)} & = 3\binom{n}{4}(1-p)^4p^2(1 - (1 - (1-p)^4)^{n-4}) \\
& = 3\binom{n}{4}(1-p)^4p^2(\sum_{j=1}^{n-4}\binom{n-4}{j}(-1)^{j-1}(1-p)^{4j}) \\
& \sim 3\binom{n}{4}(1-p)^4p^2n(1-p)^4 \sim \frac{3}{4!} n^5(1-p)^8p^2.
\end{align*}
Thus, we have that trivially $\EXP{C_4^*(n,1-p)} \to 0$ if $n(1-p)^{8/5} \to 0$. Now we use the arguments as in \eqref{e:sumeck} to show that $\sum_{k \geq 5} \EXP{C_k^*(n,1-p)} \to 0$ as well in this case. Below we will introduce an arbitrary constant $C$ whose value could change from line to line. 
\begin{align*}
\sum_{k \geq 5} \EXP{C^*_k(n,1-p)} & \leq \sum_{k \geq 5} \frac{(k-1)!}{2}\binom{n}{k}(1-p)^kp^{\binom{k}{2}-k}(1 - (1 - (1-p)^k)^{n-k}) \no \\
& \leq C  \sum_{k \geq 5} \frac{(k-1)!}{2}\binom{n}{k}(1-p)^kp^{\binom{k}{2}-k} n(1-p)^k \\
& \leq C n^5(1-p)^8p^2 \sum_{k \geq 1} \frac{(n(1-p)^2)^k p^{\binom{k+4}{2}-k-6}}{2(k+4)}\no \\
&  \leq C n^5 (1-p)^8p^2 \sum_{k \geq 1} \frac{(n(1-p)^2p^{(k+7)/2})^k}{2(k+4)}.
\end{align*}
Since $n^5 (1-p)^8 \to 0$,  the sum is finite and thus we derive that as $n \to \infty$, 
$$\sum_{k \geq 5} \EXP{C^*_k(n,1-p)}  \to 0.$$
Since $C_k^*(n,1-p) \leq C'_k( n,1-p)$, when $n\sqrt{p} \to 0$ the result follows directly from \eqref{e:Gnp4chordal} and \eqref{e:Gnpchordal}.
\end{proof}
\subsection{Proof of Theorem \ref{t:reginp}}
\label{sec:proofreg}
We first make some observations about regularity and projective dimension leading to the proofs of Lipschitz property and additivity of regularity and projective dimension. \\
 
For any ideal $I$ in a polynomial ring $S$, we have $\reg(S/I)=\reg(I)-1$. This holds as $I(G)$ by definition is the kernel of the zeroth differential of the minimal free resolution of $S/I$. Further if $G_1$ and $G_2$ are two connected components of $G$ then $(S/I(G_1)) \otimes_K (S/I(G_2)) = S/(I(G_1 \cup G_2))  = S / I(G)$. From the above observations and additivity of regularity under the tensor product $\otimes_K$ (see \cite[Lemma 2.5(ii)]{ha2016depth}), we derive that if $G_1$ and $G_2$ are two connected components of $G$ then 
\begin{equation}
\label{e:addreg}
\reg(I(G))-1 = \reg(I(G_1))-1 + \reg(I(G_2))-1.
\end{equation} 
We shall need one more observation before proving Lipschitz property of regularity. We note that if $I(G)$ is an edge ideal in $K[x_1,\ldots, x_n]$, $y$ is a variable and $J$ is the extension of $I(G)$ in $K[x_1, \ldots, x_n ,y]$ then $\reg(I(G))=\reg(J)$. Also with the same notation $\reg(J)= \reg(J +(y))$. The two assertions follow from the fact that adding redundant variables in the underlying polynomial ring does not change the minimal free resolution of a module. 
\begin{lemma}
\label{t:lipreg}
Let $G$ be a simple graph and $v \in V(G)$. Then  $|\reg(I(G)) - \reg(I(G \setminus v))| \leq 1$. 
\end{lemma}
\begin{proof}
Let us denote by $x$ the variable corresponding to vertex $v$. As a convention, we denote by $(I(G),x)$ the ideal $I(G \setminus v)+(x)$ and thus from the observations above, we obtain that $\reg(I(G),x)=\reg(I((G \setminus v))+ (x)) = \reg(I(G \setminus v))$ as $G \setminus v$ has no $x$ variable. We already know that $\reg(I(G),x) \leq \reg(I(G))$ . Now it is enough to show that $\reg(I(G),x)$ has a lower bound $\reg(I(G))-1$. Now writing $I(G)=J+xH$, where $H$ is an ideal of variables and generators of $J$ do not involve $x$, we have that $\reg(I(G)) \leq \reg(J) +1$. As $\reg(J)=\reg(J,x)= \reg(I(G),x)$, we have $\reg(I(G)) \leq \reg(I(G),x) +1$ as required.
\end{proof}

\begin{lemma}
\label{t:pdlip}
Let $G$ be a simple graph and $v \in V(G)$. Then  $|\text{pd}(I(G)) - \text{pd}(I(G \setminus v))| \leq \Delta(G)+1$ where $\Delta(G)$ is the maximum vertex degree. 
\end{lemma}

\begin{proof}
For ideals $I$ and $J$ in $S$ one has $\text{pd} (\frac{S}{I+J}) \leq \text{pd }(S/I) + \text{pd }(S/J)$. This follows from the fact that tensor product of the free resolutions give free resolution of tensor product. Now $I(G)= I(G\setminus \{x\})+ (xy_1, \ldots, xy_t)$ where the neighbours of $\{x\}$ are $\{y_1, \ldots, y_t\}$.This gives $\text{pd }(S/I(G)\leq \text{pd }S/I(G\setminus \{x\})+ \text{pd }(S/(xy_1, \ldots, xy_t))$. But $\text{pd }(S/(xy_1, \ldots, xy_t)) \leq \Delta(G)+1$. The rest follows from Theorems 3.1(i) and 4.3(ii) of \cite{Caviglia2018} and the fact that $\text{pd } I(G) + \text{depth } I(G) =  |V(G)|$.
\end{proof}
\begin{lemma}
\label{l:add_pd}
Let $G$ be any graph with $I(G)$ its edge ideal in $K[x_1, \ldots , x_n]$ and let $C_1, \ldots C_t$ be the induced subgraph on the connected components of $G$. Further let $K[C_1], \ldots, K[C_t]$ be the polynomials rings over $K$ on the vertices of $C_1, \ldots C_t$ respectively.   Then we have 
$$\text{pd} (K[x_1, \ldots x_n]/I(G)) = \sum_i \text{pd}(K[C_i]/I(G))$$. 
\end{lemma}
\begin{proof}
This follows from the fact that for polynomial rings over$k$ in disjoint set of variables $A$, $B$, graded $A$-module $M$ and graded $B$-module $N$ one has $\text{depth} (M \otimes_K N)= \text{depth}(M) + \text{depth}(N)$. And the fact that $\text{pd} (K[x_1, \ldots , x_n]/I(G)) = |V(G)|- \text{depth}(K[x_1, \ldots , x_n]/I(G))$.
\end{proof}

%
\begin{proof}[Proof of Theorem \ref{t:reginp}]
\noindent {\sc (1) :}  From \eqref{e:conc_lipfunctionals} and Lipschitz property of regularity (Lemma \ref{t:lipreg}), we obtain that
\begin{equation}
\label{e:sllnreg}
\limsup_{n \to \infty}  n^{-1}\VAR{\reg(I(n,p))} < \infty \, \, \mbox{and} \, \,\lim_{n \to \infty} \frac{\reg(I(n,p)) - \EXP{\reg(I(n,p))}}{n} = 0 \, \, \, \, \mbox{a.s..}
\end{equation}
This proves the first statement of regularity.  The claim for projective dimension follows from Lemmas \ref{l:conc_nearlip} and \ref{t:pdlip}.  From this the claim for depth also follows as $\text{depth } I(n,p) = n - \text{pd } I(n,p).$  \\ \\
\noindent {\sc (2) :} Due to the above a.s. convergence,  we can easily obtain strong laws from expectation asymptotics. for the corresponding quantities.  So,  in the below proofs,  shall focus only on the latter for $\lambda \leq 1$.  \\

Set $\reg^*(I(n,p)) := \reg(I(n,p)) - 1$. Observe that it suffices to prove the theorem for $\reg^*(I(n,p))$.  Since $\reg^*(I(n,p))$ is additive, we have that
$$ \reg^*(I(n,p)) = \sum_{i=1}^K \reg^*(I(G_i)),$$
where $K$ is the number of components of $G(n,p)$ and $G_i$ denotes the $i$th component. Suppose for $v \in [n]$, we denote the component of $v$ by $C_v(n,p)$ and $I_v(n,p)$ as the edge ideal of $C_v(n,p)$, then we can re-write the above as
$$ \reg^*(I(n,p)) = \sum_{v \in [n]} \frac{\reg^*(I_v(n,p))}{|C_v(n,p)|}.$$
Since $\frac{\reg^*(I_v(n,p))}{|C_v(n,p)|}, v \in [n]$ are identically distributed, we have that
$$\EXP{\reg^*(I(n,p))} = n \EXP{\frac{\reg^*(I_1(n,p))}{|C_1(n,p)|}},$$
and so
$$ n^{-1} \EXP{\reg^*(I(n,p))} = \EXP{\frac{\reg^*(I_1(n,p))}{|C_1(n,p)|}}.$$
Now we evaluate the RHS in the above equation by using local weak convergence of the \ER \, random graph.   Now, we assume that $\lambda \leq 1$ and study the limit of $\EXP{\frac{\reg^*(I_1(n,p))}{|C_1(n,p)|}}$.  Using that $\reg^*(I_1(n,p)) \leq |C_1(n,p)|$,  we can derive that for any $t > 0$,
\begin{align*}
\EXP{\frac{\reg^*(I_1(n,p))}{|C_1(n,p)|}} &= \EXP{\frac{\reg^*(I_1(n,p))}{|C_1(n,p)|}\1[|C_1(n,p)| \leq t]} + \EXP{\frac{\reg^*(I_1(n,p))}{|C_1(n,p)|}\1[|C_1(n,p)| > t]} \\
& \leq  \EXP{\frac{\reg^*(I_1(n,p))}{|C_1(n,p)|}\1[|C_1(n,p)| \leq t]} + \pr{|C_1(n,p)| > t}.
\end{align*}
So again using \cite[Theorem 3.12]{Bordenave2016}, we can derive
\begin{align*}
\limsup_{n \to \infty} \EXP{\frac{\reg^*(I_1(n,p))}{|C_1(n,p)|}} &\leq \limsup_{n \to \infty}\EXP{\frac{\reg^*(I_1(n,p))}{|C_1(n,p)|}\1[|C_1(n,p)| \leq t]} + \limsup_{n \to \infty}\pr{|C_1(n,p)| > t} \\
&= \EXP{\frac{\reg^*(GW(\lambda))}{|GW(\lambda)|}\1[|GW(\lambda)| \leq t]} + \pr{|GW(\lambda)| > t}.
\end{align*}
Now letting $t  \to \infty$ and using the classical fact that $|GW(\lambda)|$ is a.s. finite for $\lambda \leq 1$ \cite[Theorem 4.1]{Bordenave2016},  we have that
$$\limsup_{n \to \infty} \EXP{\frac{\reg^*(I_1(n,p))}{|C_1(n,p)|}} \leq \EXP{\frac{\reg^*(GW(\lambda))}{|GW(\lambda)|}}.$$
The lower bound can be obtained easily by the above arguments and noting that
$$ \EXP{\frac{\reg^*(I_1(n,p))}{|C_1(n,p)|}} \geq \EXP{\frac{\reg^*(I_1(n,p))}{|C_1(n,p)|}\1[|C_1(n,p)| \leq t]}.$$
Thus, we obtain that 
$$\lim_{n \to \infty} \EXP{\frac{\reg^*(I_1(n,p))}{|C_1(n,p)|}} = \EXP{\frac{\reg^*(GW(\lambda))}{|GW(\lambda)|}}.$$
Now using \eqref{e:reglimit} and the fact $\reg^*(GW(\lambda)) = \nu(GW(\lambda))$ as $GW(\lambda)$ is a tree ( \cite[Theorem 1.1]{Beyarslan2015}),  the asymptotics for regularity is complete.  \\

The statement for projective dimension $\text{pd }I(n,p)$ can be derived exactly as above using Lemma \ref{l:add_pd} and with only the simplfication in case of trees missing for projective dimension.  Now,  the claim for depth also follows as $\text{depth } I(n,p) = n - \text{pd } I(n,p).$
\end{proof}
\subsection{Proof of Theorem \ref{t:unmixed_Inp}:}
\label{sec:proofunmixed}

\begin{proof}[Proof of Theorem \ref{t:unmixed_Inp}]

By Theorem \ref{t:unmixed}, we need to show that w.h.p. there are two minimal vertex covers of different sizes for not being unmixed or show that w.h.p. any minimal vertex cover is of same size to prove unmixedness. Since complements of vertex covers are independent sets which are equivalent to cliques in the complement graph, it suffices to also show that w.h.p. there are maximal cliques of different sizes in $G(n,1-p)$ and conclude that $I(n,p)$ is not unmixed. 

\begin{enumerate}
\item By Theorem \ref{t:denrg}, the components of $G(n,p)$ are either singleton edges or  singleton vertices w.h.p.. Thus any minimal vertex cover consists of one vertex from each edge component and hence the size of a minimal vertex cover is always the same. 

\item If $\alpha >1$, by Theorem \ref{t:denrg}, $G(n,p)$ will have a path component on $3$ vertices w.h.p.. Every path component has two minimal vertex covers, one of size $1$ and the other of size $2$. If $\alpha = 1$, we know from \cite[Corollary 3.24]{Bordenave2016} that $G(n,p)$ will have any finite tree as a component w.h.p.. Hence, again the minimal vertex cover is not unique.

\item When $p \in (0,1)$, there are maximal cliques of different orders in $G(n,1-p)$ w.h.p.; see \cite[p. 424]{Bollobas1976cliques}. 

\item Let $1 - p = n^{-\alpha}$ and $\alpha^{-1} \notin \{2,3,\ldots,\infty\}$. We will argue differently in the cases $\alpha <1$ and $\alpha > 1$. \\

First, let $\alpha < 1$. Then there exists a $k \geq 2$ such that $\frac{1}{k} < \alpha < \frac{1}{k-1}$. Hence, we know that there exist maximal $k$-cliques in $G(n,1-p)$ w.h.p.  \cite[Lemma 2.2]{Kahle2014sharp}. Further, since $\alpha < \frac{1}{k-1} \leq \frac{2}{k}$, there exist $(k+1)$-cliques in $G(n,1-p)$ w.h.p. (see Theorem \ref{t:denrg} ). So there are maximal cliques of order $k$ and cliques of an order at least $(k+1)$. Thus, there are maximal cliques of different orders and so $I(n,p)$ is not unmixed. \\

Now let $\alpha \geq 1$. Again, by Theorem \ref{t:denrg} and 
\cite[Corollary 3.24]{Bordenave2016}, there are both tree components and isolated vertices in $G(n,1-p)$ w.h.p.. Thus there are maximal cliques of order $1$ and $2$ and so $I(n,p)$ is not unmixed. 

\item By Theorem \ref{t:denrg}, again $G(n,1-p)$ has no edges w.h.p and so $G(n,p)$ is complete w.h.p.. Hence minimal vertex covers are unique and so $I(n,p)$ is unmixed. 
\end{enumerate}
\end{proof}
\section*{Acknowledgements}
D.Y. was supported by DST INSPIRE Faculty award,  SERB-MATRICS grant and CPDA from the Indian Statistical Institute.  A.B was supported by DST INSPIRE Faculty award and CPDA from the Ramakrishna Mission Vivekananda Educational and Research Institute.  We are extremely grateful to Prof. Daniel Erman for his detailed comments and suggestions on earlier drafts as well as pointing out references \cite{deLoera2019average,booms2020heuristics}.  A.B. would like to thank Prof. B. V. Rao,  Prof. Huy Tai Ha and Prof. Giulio Caviglia for their valuable suggestions. A.B. would also like to thank Indian Statistical Institute, Bengaluru centre for the hospitality during his visit when this work was partially done.

\bibliographystyle{plain}
\bibliography{random_edge_ideals}

\end{document}